\documentclass [12pt]{amsart}
\usepackage[utf8]{inputenc}
\pdfoutput=1
\usepackage{Andrew}

\usepackage{graphicx}
\usepackage{amsmath,amssymb,amsthm,amsfonts}

\usepackage{hyperref}%

\usepackage{thmtools}
\usepackage{subcaption}
\hypersetup{bookmarksdepth=2}
\usepackage{mathrsfs}%

\usepackage{xcolor}%
\usepackage{xspace}

\usepackage[margin=1.0in]{geometry}%
\usepackage{enumitem}

\usepackage{hyperref}

\usepackage{epstopdf}%

\newtheorem{theorem}{Theorem}[section]

\newtheorem{lemma}[theorem]{Lemma}

\newtheorem{corollary}[theorem]{Corollary}
\theoremstyle{definition}

\theoremstyle{definition}

\theoremstyle{definition}
\newtheorem{definition}[theorem]{Definition}

\newtheorem{question}[theorem]{Question}
\theoremstyle{definition}

\theoremstyle{definition}

\newtheorem{observation}[theorem]{Observation}

\def\Tbad{T_{\text{bad}}}
\def\Tgood{T_{\text{good}}}

\def\TU{T_{U}}
\def\TL{T_{L}}
\def\op{\text{op}}
\def\append{\cdot}

\def\flipmap{\Phi_{P,S}}
\def\flipmapinverse{\Phi_{P',S}}

\title{The poset associahedron $f$-vector is a comparability invariant}
\author{Son Nguyen and Andrew Sack}
\address{School of Mathematics, University of Minnesota, Minneapolis, MN 55455}
\email{{\href{mailto:nguy4309@umn.edu}{nguy4309@umn.edu}}}
\address{Department of Mathematics, University of California, Los Angeles, CA 90095, USA}
\email{{\href{mailto:andrewsack@math.ucla.edu}{andrewsack@math.ucla.edu}}}

\thanks{This material is based upon work supported by the National Science Foundation
Graduate Research Fellowship Program under Grant No. DGE-2034835 and National Science Foundation Grants No. DMS-1954121 and DMS-2046915. Any
opinions, findings, and conclusions or recommendations expressed in this material
are those of the author(s) and do not necessarily reflect the views of the National
Science Foundation.}
\date{\today}

\usepackage[backend=bibtex]{biblatex}
\begin{document}


\begin{abstract}
We show that the $f$-vector of Galashin's poset associahedron $\mathscr A(P)$ only depends on the comparability graph of $P$.  In particular, this allows us to produce a family of polytopes with the same $f$-vectors as permutohedra, but that are not combinatorially equivalent to permutohedra.
\end{abstract}

\maketitle

\section{Introduction}

Recall that the \emph{comparability graph} of a poset $P$ is a graph $C(P)$ whose vertices are the elements of $P$ and where $i \tand j$ are connected by an edge if $i \tand j$ are comparable. A property of $P$ is said to be \emph{comparability invariant} if it only depends on $C(P)$.  Properties of finite posets known to be comparability invariant include the order polynomial and number of linear extensions~\cite{stanley1986two}, the fixed point property~\cite{dreesen1985comparability}, and the Dushnik–Miller dimension~\cite{trotter1976dimension}.  

For a finite connected poset $P$, Galashin introduced the \emph{poset associahedron} $\mathscr A(P)$~\cite{galashin2021poset}.  Poset associahedra generalize Stasheff's associahedron~\cite{StasheffCyclohedron} to the setting of properadic composition.  That is, instead of having a sequence of operations with one input and one output, one may view a poset as a collection of operations with multiple inputs and multiple outputs ``wired together'' by covering relations.  A vertex of a poset associahedron disambiguates the order of composition.  For more details, see~\cite{laplante2022diagonal, sack2023realization, stoeckl2023koszul}. 

The $f$-polynomial of a $d$-dimensional polytope $Q$ is $$f_Q(z) := \sum\limits_{i = 0}^d f_iz^i$$ where $f_i$ is the number of faces of $Q$ in dimension $i$. We call $(f_0, \dots, f_d)$ the $f$-vector of $Q$. The following is our main result:
\begin{theorem}
\label{thm:main_thm}
The $f$-vector of $\mathscr A(P)$ is a comparability invariant.
\end{theorem}


Theorem \ref{thm:main_thm} may lead one to ask if $C(P) \simeq C(P')$ implies that $\mathscr A(P) \tand \mathscr A(P')$ are necessarily combinatorially equivalent.

\begin{definition}
Let $\mathbf a = (a_1, \dots, a_n) \in \Z^n$ with $a_i \ge 1$ for each $i$.  Define the \emph{complete graded poset} of type $\mathbf{a}$ to be the poset $$P_{\mathbf{a}} := \{x_{11}, \dots, x_{1a_1}, x_{21}, \dots, x_{2a_{2}}, \dots\},$$ where $x_{ij} \prec x_{i'j'}$ if and only if $i < i'$. That is, $P_{\mathbf a}$ is the ordinal sum of antichains.
\end{definition}

Observe that $C(P_{\mathbf{a}})$ is invariant under permutation of $\mathbf{a}$.  This observation, together with Theorem~\ref{thm:main_thm}, yields an immediate corollary.

\begin{corollary}
    For any $\mathbf a$, $f_{\mathscr A(P_{\mathbf a})}(z)$ is invariant under permutation of $\mathbf a$.
\end{corollary}

This class of examples is sufficiently rich to answer our question in the negative.


\begin{theorem}
    \label{thm:combinatorially_equivalent}
    Let $m ,n \ge 2$.  Then
    $\mathscr A(P_{(m, 1, n)})$ is combinatorially equivalent to the permutohedron, but $\mathscr A(P_{(1, m, n)})$ is not.
\end{theorem}

\section{Background}

\subsection{Flips of autonomous subsets}

\begin{definition}
    Let $P \tand S$ be posets and let $a \in P$.  The \emph{substitution} of $a$ for $S$ is the poset $P(a \to S)$ on the set $(P-\{a\}) \sqcup S$ formed by replacing $a$ with $S$.  
    
    More formally, $x \preceq_{P(a \to S)} y$ if and only if one of the following holds:
    \begin{itemize}
        \item $x, y \in P-\{a\}$ and $x \preceq_P y$
        \item $x, y \in S$ and $x \preceq_S y$
        \item $x \in S, y \in P-\{a\}$ and $a \preceq_P y$
        \item $y \in S, x \in P-\{a\}$ and $y \preceq_P a$.
    \end{itemize}

\end{definition}

\begin{definition}
    \label{def:auto_subset}
Let $P$ be a poset and let $S \subseteq P$.  $S$ is called \emph{autonomous} if there exists a poset $Q$ and $a \in Q$ such that $P = Q(a \to S)$.  

Equivalently, $S$ is autonomous if  $\text{for all } x, y \in S \tand z \in P-S, \text{ we have }$ $$(x \preceq z \Leftrightarrow y \preceq z) \tand (z \preceq x \Leftrightarrow z \preceq y).$$ 
\end{definition}

    
    


\begin{definition}
    \label{def:flip}
    For a poset $S$, the \emph{dual poset} $S^{\op}$ is defined on the same ground set where $x \preceq_S y$ if and only if $y \preceq_{S^{\op}} x$. A \emph{flip} of $S$ in $P = Q(a \to S)$ is the replacement of $P$ by $Q(a \to S^{\op})$. 

\begin{figure}[h!]
    \centering
    \begin{subfigure}[t]{0.45\textwidth}
    \centering
        \includegraphics[scale = 1.2]{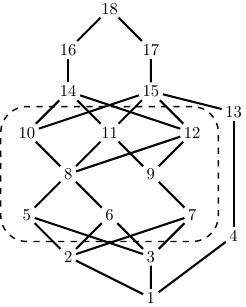}
    \caption{An autonomous subset $S$ of a poset $P$.}
    \label{fig:AutonomousExample}
        
    \end{subfigure}~
    \begin{subfigure}[t]{0.45\textwidth}
        \centering
        \includegraphics[scale = 1.2]{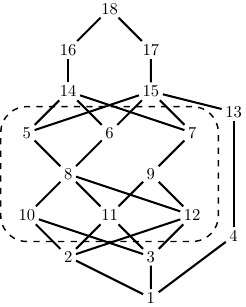}
        \caption{A flip of $S$.}
        \label{fig:FlipExample}
    \end{subfigure}
    \caption{}
\end{figure}

See Figure \ref{fig:AutonomousExample} for an example of an autonomous subset and Figure \ref{fig:FlipExample} for an example of a flip.

\end{definition}

\begin{lemma} [{\cite[Theorem 1]{dreesen1985comparability}}] 
\label{lem:dreesen_comparability}

If $P \tand P'$ are finite posets such that $C(P) = C(P')$ then $P \tand P'$ are connected by a sequence of flips of autonomous subsets. 
    
\end{lemma}

In particular, a property is comparability invariant if and only if it is preserved under flips.

\subsection{Poset Associahedra}

We recall the definition of poset associahedra.

\begin{definition}
Let $P$ be a finite connected poset.  A subset $\tau \subsetneq P$ is called a \emph{proper tube} if
    \begin{itemize}
        \item $2 \le |\tau|$.
        \item $\tau$ is \emph{convex}, i.e. for all $x, z \in \tau, y \in S$, we have 
        $$(x \preceq y \preceq z) \Rightarrow (y \in \tau).$$
        \item $\tau$ is \emph{connected} as a subgraph of the Hasse diagram of $P$.
    \end{itemize}
A collection $T$ of proper tubes is called a \emph{proper tubing} if the following two conditions hold:
    \begin{itemize}
        \item The tubes in $T$ are pairwise either nested or disjoint.  That is, for all $\sigma, \tau \in T$, we have either $\sigma \subseteq \tau, \tau \subseteq \sigma, \tor \tau \cap \sigma = \emptyset.$
        \item The directed graph $D_T$ is acyclic, where $T$ is the vertex set of $D_T$ and where $(\sigma, \tau)$ is an edge if $\sigma \cap \tau = \emptyset$ and there exist $x \in \sigma, y \in \tau \st x \prec y$.
    \end{itemize}

\end{definition}

See Figure \ref{fig:ProperTubing} for an example of a proper tubing.

\begin{figure}[h!]
\captionsetup[subfigure]{justification=centering}

    \centering
    \begin{subfigure}{0.45\textwidth}
    \centering
        \includegraphics[scale = 1]{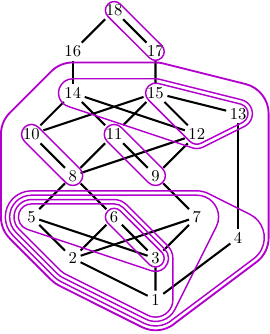}
    \caption{A proper tubing $T$ on $P$.}
    \label{fig:ProperTubing}
    \end{subfigure}~
    \begin{subfigure}{0.45\textwidth}
        \centering
        \includegraphics[scale = 1]{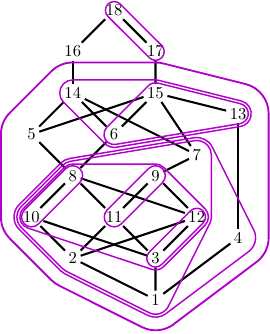}
        \caption{$\flipmap(T)$}
        \label{fig:TprimeFinal}
    \end{subfigure}
    \caption{}
    \label{fig:TubingAndImageUnderMap}
\end{figure}

\begin{theorem}[{\cite[Theorem 1.2]{galashin2021poset}}]
    Let $P$ be a finite, connected poset on at least 2 elements.  Then the collection of proper tubings of $P$ ordered by reverse inclusion is isomorphic to the face lattice of a simple $(|P|-2)$-dimensional polytope $\mathscr A(P)$.  We call this polytope a poset associahedron.
\end{theorem}

\begin{lemma}[{\cite[Corollary 2.7]
{galashin2021poset}}]
    The codimension of $T \in \mathscr A(P)$ is equal to $|T|$.
\end{lemma}

By an abuse of notation, we also use $\mathscr A(P)$ to refer to the set of proper tubings of $P$.  Our strategy for proving Theorem~\ref{thm:main_thm} is to give a bijection between the tubings of $Q(a \to S)$ and of $Q(a \to S^{\op})$ that preserves the number of tubes in a tubing.  See Figure \ref{fig:TubingAndImageUnderMap} for an example of the map.

\section{Proof of Theorem \ref{thm:main_thm}}

\subsection{Proof Sketch}

Let $P = Q(a \to S)$ and $P' = Q(a \to S^{\op})$. Our goal is to build a bijection $\flipmap: \mathscr A(P) \to \mathscr A(P')$ such that for any $T \in \mathscr A(P), |T| = |\flipmap(T)|$. Let $T \in \mathscr A(P)$.  We will describe how to construct $T' := \flipmap(T)$.

\begin{definition}
A tube $\tau \in T$ is \emph{good} if $\tau \subseteq P - S$, $\tau \subseteq S, \tor S \subseteq \tau$ and is \emph{bad} otherwise.  We denote the set of good tubes by $\Tgood$ and the set of bad tubes by $\Tbad$.  
    
\end{definition}





\begin{figure}[h!]
    \centering
    \includegraphics[scale = 1]{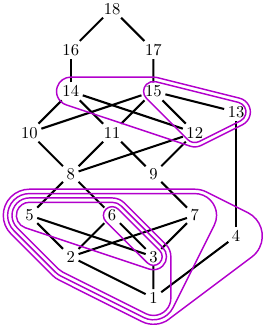}
    \includegraphics[scale = 1]{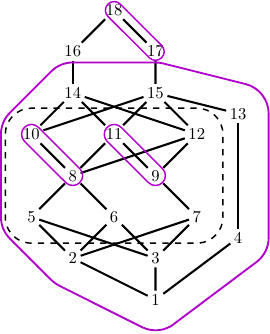}
    \includegraphics[scale = 1]{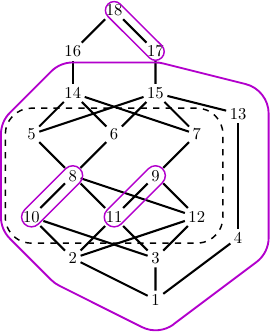}
    \caption{$\Tbad$ (left), $\Tgood$ (middle), and $\Tgood$ on $P'$ (right).}
    \label{fig:TGoodAndTBad}
\end{figure}

All good tubes are also good tubes in $P'$, and we add all good tubes to $T'$. See Figure \ref{fig:TGoodAndTBad} for an example of $\Tgood \tand \Tbad$.  It remains to handle the bad tubes.

\begin{definition}
        A sequence of sets $(A_1, \dots, A_r)$ is called \emph{nested} if $A_i \subseteq A_j$ for all $i \le j$.  
    A \emph{decorated} nested sequence is a nested sequence $(A_1, \dots, A_r)$ paired with a function
    $$f: \{1, \dots, r\} \to \{0, 1\}.$$ 
    For brevity, instead of specifying $f$, we will instead mark $A_i$ with a star if and only if $f(i) = 1$.
\end{definition}

The key idea of defining $\flipmap$ is to decompose $\Tbad$ into a triple $(\mathcal L, \mathcal M, \mathcal U)$ where $\mathcal L \tand \mathcal U$ are decorated nested sequences of sets contained in $P-S$ and $M$ is an ordered set partition of $S$.  We build  the decomposition in such a way so that we can recover $\Tbad$ from $(\mathcal L, \mathcal M, \mathcal U)$ and Figure \ref{fig:decomposition} for an example of the decomposition.

We build $\Tbad'$ by applying the recovery algorithm to the triple $(\mathcal L, \overline{\mathcal M}, \mathcal U)$ where $\overline{\mathcal M}$ is the reverse of $\mathcal M$.  We then add $\Tbad'$ to $T'$.  See Figure \ref{fig:flipmap} for an example of the recovery algorithm applied to $(\mathcal L, \overline{\mathcal M}, \mathcal U)$.  See Figure \ref{fig:TprimeFinal} for the image of $T$ under $\flipmap$ (including $\Tgood$).

\subsection{Proof details}

\begin{definition}
    A tube $\tau \in \Tbad$ is called \emph{lower} (resp. \emph{upper}) if there exist $x \in \tau-S$ and $y \in \tau \cap S \st x \preceq y$ (resp. $y \preceq x$).  We denote the set of lower tubes by $\TL$ and the set of upper tubes by $\TU$.
\end{definition}


\begin{lemma}[Structure Lemma]
    $\Tbad$ is the disjoint union of $\TL$ and $\TU$.  
    Furthermore, $\TL \tand \TU$ each form a nested sequence.
\end{lemma}
\begin{proof}
We first show that $\Tbad$ is the disjoint union of $\TL$ and $\TU$.
Suppose that $\tau \in \TL \cap \TU$, i.e. there exist $x_1, x_2 \in \tau - S \tand y_1, y_2 \in \tau \cap S$ such that $$x_1 \preceq y_1 \tand y_2 \preceq x_2.$$  Then as $S$ is autonomous, for all $y \in S, x_1 \preceq y \preceq x_2$.  As $\tau$ is convex, this implies $S \subseteq \tau$ and hence that $\tau$ is good.  Therefore $\TL \tand \TU$ are disjoint.  Next observe that if $\tau \in \Tbad$, by connectivity there exist $x \in \tau \cap S \tand y \in \tau-S$ such that $x \tand y$ are comparable.  Hence $\tau \in \TL \cup \TU$ so $\Tbad = \TL \sqcup \TU$.

Finally, we show that $\TL$ is nested.  The result on $\TU$ follows analogously.  It suffices to show that $\TL$ is pairwise nested.  Let $\sigma, \tau \in \TL$.  As $T$ is a tubing, if $\sigma \tand \tau$ are not nested, then they are disjoint.  Suppose, for the sake of contradiction, that $\sigma \cap \tau = \emptyset$, and let $x_1 \in \tau - S, x_2 \in \sigma-S, y_1 \in \tau \cap S, \tand y_2 \in \sigma \cap S \st x_1 \preceq y_1 \tand x_2 \preceq y_2$.  Then as $S$ is autonomous, $x_1, x_2 \preceq y_1, y_2$.  Thus $(\sigma, \tau) \tand (\tau, \sigma)$ are both edges in $D_T$, so $D_T $ is not acyclic, a contradiction.
\end{proof}

\begin{figure}
    \centering
    \begin{subfigure}[t]{0.33\textwidth}
    \centering
    \includegraphics[scale = 1.2]{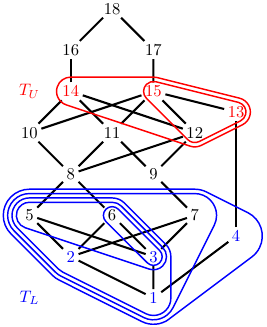}
    \caption{$T_L$ is blue and $T_U$ is red.    }
    \end{subfigure}~
    \begin{subfigure}[t]{0.66\textwidth}
    \centering
    \includegraphics[scale = 1.2]{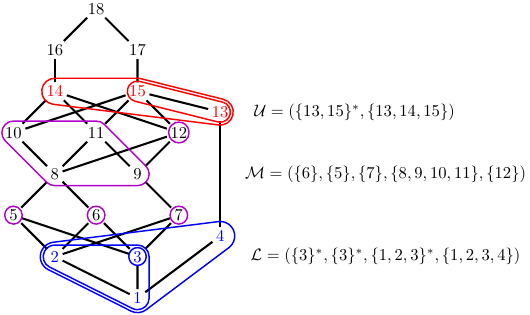}
    \caption{$\mathcal L$ is blue, $\mathcal M$ is purple, and $\mathcal U$ is red.}
    \end{subfigure}
    \caption{The decomposition of $\Tbad$.}
    \label{fig:decomposition}

\end{figure}

We decompose $\TL$ (resp. $\TU$) into a sequence of nested sets contained in $P-S$ and a sequence of disjoint sets contained in $S$ as follows.

\begin{definition}[Tubing decomposition]  
\label{def:tubing_decomp}

    Let $\TL = \{\tau_1, \dots, \}$ where $\tau_i \subset \tau_{i+1}$ for all $i$.  For convenience, we define $\tau_0 = \emptyset$. We define a decorated nested sequence $\mathcal L = (L_1, \dots)$ and a sequence of disjoint sets $\mathcal M_L = (M_{L}^1, \dots)$ as follows. 
    \begin{itemize}

    \item For each $i \ge 1$, let $L_i = \tau_i - S$, and mark $L_i$ with a star if $(\tau_i - \tau_{i-1}) \cap S \neq \emptyset$.  

\item If $L_{i}$ is the $j$-th starred set, let $M_{L}^j = (\tau_{i}-\tau_{i-1}) \cap S$.

\end{itemize}
We define the sequences $\mathcal U$ and $\mathcal M_U$ analogously.  
We make the following definitions.
\begin{itemize}
    
    \item Let $\hat M := S - \bigcup\limits_{\tau \in \Tbad} \tau.$

    \item For sequences $\mathbf{a} \tand \mathbf{b}$, let the sequence $\mathbf{a} \append \mathbf{b}$ be $\mathbf{b}$ appended to $\mathbf{a}$.
    \item For a sequence $\mathbf{a}$, let $\overline{\mathbf a}$ be the reverse of $\mathbf{a}$.

    \item We define $$\mathcal M :=  \begin{cases}
        \mathcal M_L \append \overline{\mathcal M}_U & \text{if } \hat M = \emptyset\\

        \mathcal M_L \append (\hat M) \append \overline{\mathcal M}_U & \text{if } \hat M \neq \emptyset 
        
    \end{cases}$$
    where $(\hat M)$ is the sequence containing $\hat M$.

\item The \emph{decomposition} of $\Tbad$ is the triple $(\mathcal L, \mathcal M, \mathcal U)$.
\end{itemize}

\end{definition}

See Figure \ref{fig:decomposition} for an example a decomposition.

\begin{lemma}[Reconstruction algorithm]
    $\Tbad$ can be reconstructed from its decomposition.
\end{lemma}
\begin{proof}
    Let $\mathcal M = (M_1, \dots, M_n)$. To reconstruct $\TL$, we set $\tau_1 = L_1 \cup M_{1}$ and take
$$\tau_i =\begin{cases} \tau_{i-1} \cup L_i & \text{if $L_i$ is not starred}
\\ \tau_{i-1} \cup L_i \cup M_{j} & \text{if $L_i$ is marked with the $j$-th star.} \end{cases}$$

For $\TU$, we set $\tau_1 = U_1 \cup M_{n}$ and
$$\tau_i =\begin{cases} \tau_{i-1} \cup U_i & \text{if $U_i$ is not starred}
\\ \tau_{i-1} \cup U_i \cup M_{n-j+1} & \text{if $U_i$ is marked with the $j$-th star.} \end{cases}$$ 

In each case, the efficacy of the algorithm follows easily from induction on $i$.
\end{proof}

\begin{definition}[Flip map for tubings]

Let $T = \Tgood \sqcup \Tbad$. The \emph{flip map} $$\flipmap: \mathscr A(P) \to \mathscr A(P')$$ sends $T$ to a tubing $T' = \Tgood' \sqcup \Tbad'$ on $P'$ where $\Tgood = \Tgood'$ and $\Tbad'$ has the decomposition $(\mathcal L, \overline{\mathcal M}, \mathcal U)$.
\end{definition}

In Lemma \ref{lem:TbadIsATubing}, we show that applying the reconstruction algorithm to $(\mathcal L, \overline{\mathcal M}, \mathcal U)$ indeed yields a proper tubing $\Tbad'$ of $P'$.  In Lemma \ref{lem:WellDefined}, we show that $\Tgood \sqcup \Tbad'$ is a proper tubing on $P'$ and hence that $\flipmap$ is well-defined.

\begin{observation}
\label{obs:bijection}
By construction, the decomposition of $\Tbad'$ is $(\mathcal L, \overline{\mathcal M}, \mathcal U)$, so applying $\flipmapinverse$ returns $T$. In particular, $\flipmap$ is a bijection.    
\end{observation}

\begin{figure}[h!]
    \centering
        \includegraphics[scale = 1.2]{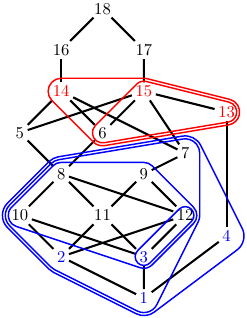}
    \includegraphics[scale = 1.2]{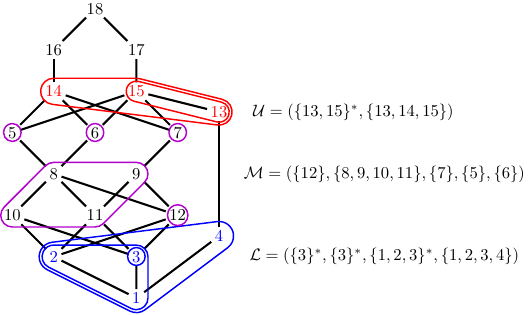}
    \caption{$\Tbad'$ and its decomposition.}
    \label{fig:flipmap}
\end{figure}


\begin{definition}
 Let $\mathbf{A} = (A_1, \dots, A_n)$ be a sequence of disjoint subsets of $P$.  We say $\mathbf{A}$ is \emph{weakly increasing} if for all $i < j$ we have $(x \in A_i \tand y \in A_j) \Rightarrow y \not\prec x.$

\end{definition}


\begin{lemma}
    $\mathcal M$ is weakly increasing.  
\end{lemma} 

\begin{proof}
    First we show that $\mathcal M_L$ is weakly increasing.  Indeed, suppose to the contrary that $1 \le i < j \le |\mathcal M_L|$ but that there exist $x \in M_i \tand y \in M_j \st y \prec x$.  As $i < j$, there exists a tube $\tau \in \TL \st x \in \tau$ but $y \notin \tau$.  Furthermore, as $\tau$ is a lower tube, there exists $z \in\tau - S \st z \preceq y$.  Then since $\tau$ is convex, $y \in \tau$, a contradiction.

    Next, we show that $\mathcal M_L \append \hat M$ is weakly increasing. Let $x \in M_i \tand y \in \hat M$ such that $1 \le i \le |\mathcal M_L|$. 
    Then there exists a tube $\tau \in T_L \st x \in T_L$.  Again, there exists $z \in\tau - S \st z \preceq y$.  Then by the same convexity argument, if $y \prec x$ we have $y \in \tau$, contradicting the definition of $\hat M$. Hence $\mathcal M_L \append \hat M$ is weakly increasing.

    By symmetry, we have that $\hat M \append \overline{\mathcal M_U}$ is weakly increasing.  It remains to show that for all $x \in \bigcup_{A \in \mathcal M_L} A \tand y \in \bigcup_{A \in \mathcal M_U} A$ we have $y \not\prec x.$

    Suppose to the contrary that there are such $x \tand y$.  Then there exist $\sigma \in T_L \tand \tau \in T_U$  with $x \in \sigma \tand y \in \tau$.  Furthermore, there exist $a \in \sigma \tand b \in \tau \st a \preceq x \tand y \preceq b.$  But then we have a cycle in $D_T$, a contradiction.

\end{proof}

\begin{lemma}
\label{lem:TbadIsATubing}
    $\Tbad'$ is a proper tubing on $P'$ such that $|\Tbad'| = |\Tbad|$.
\end{lemma}
\begin{proof}
    By construction, for all $\sigma, \tau \in \Tbad'$, $\sigma \tand \tau$ are nested or disjoint.  Furthermore, observe that in the construction of $\TL' = (\tau_1', \dots)$, if $L_i$ is empty then it is necessarily starred.  Thus for all $i$, we have $\tau_i' \subsetneq \tau_{i+1}'$.   Then $|\TL'| = |\mathcal L| =|\TL|$.  Similarly, $|\TU'| = |\TU|$. Hence $$|\Tbad'| =|\TL| + |\TU| = |\Tbad|.$$

    It remains to show that $D_{\Tbad'}$ is acyclic.  It suffices to show that $A := \bigcup_{\tau' \in \TL'} \tau'$ and that $B:= \bigcup_{\tau' \in \TU'} \tau'$ do not form a directed cycle.  Observe that as $\mathcal M$ is weakly increasing in $P$,  $\overline{\mathcal M}$ is weakly increasing in $P'$.
    Hence $(A, B)$ is weakly increasing, so $A \tand B$ do not form a directed cycle.
\end{proof}

\begin{lemma}
$\Tgood \sqcup \Tbad' $ is a proper tubing on $P'$.
\label{lem:WellDefined}
\end{lemma}
\begin{proof}
   This is most easily seen by observing how $\flipmap$ interacts with quotients of good tubes.
   Galashin~\cite[Corollary 2.7]{galashin2021poset} observes that faces of poset associahedra are products of poset associahedra.  In particular, given $T \in \mathscr A(P)$ and $\tau \in T \cup \{P\}$, we define an equivalence relation $\sim_\tau$ on $\tau$ by $i \sim_\tau j$ if there exists $\sigma \in T$ such that $i, j \in \sigma$ and $\sigma \subsetneq \tau$. Then the facet corresponding to $T$ is combinatorially equivalent to the product $\prod_{\tau \in T\cup\{P\}} \mathscr A(\tau/{\sim}_\tau)$.

   Let $\tau \in \Tgood \cup \{P\}$ be minimal such that $S \subseteq \tau$. One may verify that $\flipmap$ on any tubing containing $\Tgood$ is equivalent to applying $\Phi_{T/{\sim}_{\tau}, S/{\sim}_{\tau}}$ on the factor of $T/{\sim}_\tau$ in the product decomposition.  
   Then either $\Tgood = \emptyset$ and $\flipmap$ is well-defined by Lemma \ref{lem:TbadIsATubing} or $\flipmap$ is well-defined by induction on the size of $P$.
\end{proof}

We can finally prove Theorem~\ref{thm:main_thm}.  
\begin{proof}[Proof of Theorem~\ref{thm:main_thm}]

By Observation \ref{obs:bijection}, $\flipmap: \mathscr A(P) \to \mathscr A(P')$ is a bijection.  Furthermore, for any tubing $T \in \mathscr A(P)$, we have $$|\flipmap(T)| = |\Tbad| + |\Tgood| =|T|.$$ Hence the $f$-vectors of $\mathscr A(P) \tand \mathscr A(P')$ are equal.  By Lemma~\ref{lem:dreesen_comparability}, the $f$-vector of $\mathscr A(P)$ is a comparability invariant.
    
\end{proof}

\section{Proof of Theorem \ref{thm:combinatorially_equivalent}}

\begin{observation}[\cite{laplante2022diagonal, mantovani2023Poset}]
\label{obs:graph_ass}
If the Hasse diagram of $P$ is a tree, then $\mathscr A(P)$ is combinatorially equivalent to the graph associahedron~\cite{postnikov2008faces} of the line graph of the Hasse diagram of $P$.    
\end{observation}

\begin{proof}[Proof of Theorem \ref{thm:combinatorially_equivalent}]

By Observation \ref{obs:graph_ass}, for any $m, n \ge 1$, $\mathscr A(P_{m, 1, n})$ is combinatorially equivalent to the permutohedron $\Pi_{m+n}$.  

However, for $m, n \ge 2$, $\mathscr A(P_{1, m, n})$ has an octagon for a 2-dimensional face which permutohedra never do.  In particular, an octagon is a factor of the facet given by any tube isomorphic to $P_{2,2}$.
\end{proof}

\begin{figure}[h!]
    \centering
    \begin{tabular}{ccc}

        \includegraphics[height = 5cm]{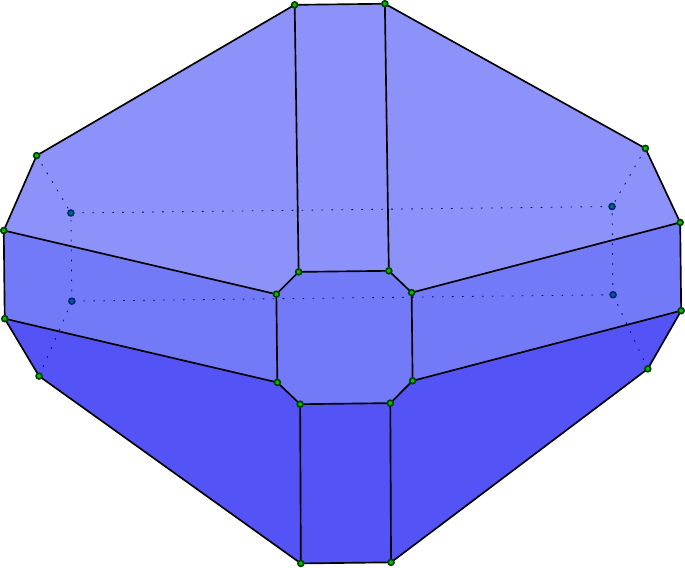} & \qquad\qquad & \includegraphics[height = 5cm]{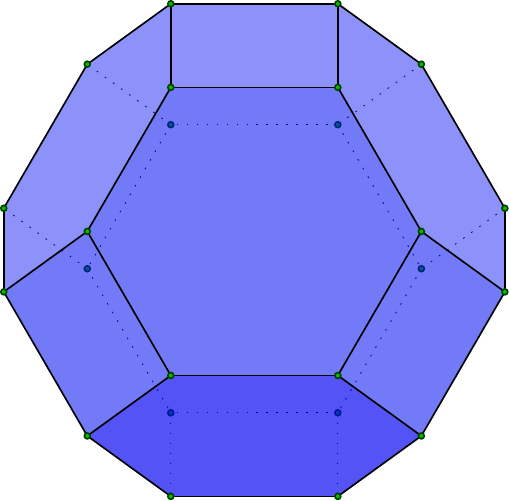} \\

                 $\mathscr A(P_{(1,2,2)})$ & \qquad\qquad & $\mathscr A(P_{(2,1,2)})$ 
    \end{tabular}
    \caption{$\mathscr A(P_{(1,2,2)})$ has an octagonal face, but $\mathscr A(P_{(2,1,2)})$
    does not.}
    \label{fig:counterexample}
\end{figure}

\section{Open questions}

\begin{question}
    In \cite{stanley1986two}, Stanley defines the \emph{order polytope} and the \emph{chain polytope}, with the latter defined purely in terms of the comparability graph.  He constructs a piecewise linear volume preserving map between the two polytopes which sends vertices to vertices.

    In particular, this shows that the number of vertices of the order polytope is a comparability invariant.  Can a similar geometric map be defined on the realization of poset associahedra in \cite{sack2023realization}?

\end{question}

\begin{question}
    More generally, can we define  $f_{\mathscr A(P)}(z)$ purely in terms of $C(P)$?  It would also be interesting to answer this question even for $f_0$.
\end{question}

\begin{question}
    It remains open to find an interpretation of $$h_{\mathscr A(P)}(z) := f_{\mathscr A(P)}(z-1)$$ in terms of the combinatorics of $P$.  Can $h(z)$ be defined purely in terms of $C(P)$?
\end{question}

\begin{question}
    The flip map can be analogously defined for \emph{affine poset cyclohedra}~\cite{galashin2021poset}, where an autonomous subset $S$ has at most one representative from each residue class.  Again, it preserves the $f$-vector of the affine poset cyclohedron.
    Does Lemma \ref{lem:dreesen_comparability} (and hence Theorem~\ref{thm:main_thm}) hold for affine posets?
\end{question}

\section*{Acknowledgements}
The authors are grateful to Vic Reiner for bringing Lemma \ref{lem:dreesen_comparability} to their attention and to Pavel Galashin for his helpful remarks.

\printbibliography

\end{document}